\documentclass[11pt]{article}

\usepackage{graphicx}
\usepackage[margin=0.5in]{geometry}
\usepackage{bm}
\usepackage{amsthm}
\usepackage{amscd}
\usepackage{amsbsy}
\usepackage{appendix}
\usepackage{amsmath}
\usepackage{amsfonts}
\usepackage{amssymb}
\usepackage{calligra}
\usepackage{xcolor}
\usepackage{float}
\usepackage[T1]{fontenc}
\usepackage{multirow}
\usepackage{mathrsfs}
\usepackage{mathtools}
\usepackage{epstopdf}
\usepackage{psfrag}
\usepackage{subfigure}
\usepackage{booktabs}
\usepackage{listings}
\usepackage{longtable}
\usepackage{latexsym}
\usepackage{arydshln}
\usepackage{enumerate}
\usepackage{epsfig}
\usepackage{cite}
\usepackage{verbatim}
\usepackage{overpic}
\usepackage{abstract}
\usepackage{extarrows}
\usepackage[pagewise]{lineno}

\allowdisplaybreaks[4]

\date{}
\title{Integrability and Linearizability of a Family of Three-Dimensional Polynomial Systems}
\author{Bo Huang, Ivan Mastev and Valery Romanovski\footnote{Corresponding author.}\\
	\it\footnotesize LMIB -- School of Mathematical Sciences,
Beihang University, Beijing 100191, China \\
	\it\footnotesize bohuang0407@buaa.edu.cn\\
	\it\footnotesize Center for Applied Mathematics and Theoretical Physics, SI-2000 Maribor, Slovenia\\
    \it\footnotesize ivan.mastev@student.um.si\\
    \it\footnotesize Center for Applied Mathematics and Theoretical Physics, SI-2000 Maribor, Slovenia\\
    \it\footnotesize Faculty of Natural Science and Mathematics, University of Maribor, SI-2000 Maribor, Slovenia\\
    \it\footnotesize Faculty of Electrical Engineering and Computer Science, University of  Maribor, SI-2000 Maribor, Slovenia\\
	\it\footnotesize valerij.romanovskij@um.si}

\newtheorem {theorem*}{Theorem}
\newtheorem{theorem} {Theorem}

\newtheorem{definition}{Definition}

\newtheorem{lemma}{Lemma}

\newtheorem{remark}{Remark}

\newtheorem{conjecture} {Conjecture}

\newtheorem{open problem} {Open problem}

\numberwithin{equation}{section}

\usepackage[colorlinks,
            CJKbookmarks=true,
            linkcolor=blue,
            anchorcolor=red,
             urlcolor=blue,
            citecolor=blue
            ]{hyperref}

\begin{document}
\maketitle
\noindent {\bf Abstract.} We investigate the local integrability and linearizability of a family of three-dimensional polynomial systems with the matrix of the linear approximation having the eigenvalues $1, \zeta, \zeta^2 $, where $\zeta$ is a primitive cubic root of unity.  We establish a criterion for the convergence of the Poincar\'e--Dulac normal form of the systems and examine the relationship between the normal form and integrability. Additionally, we introduce an efficient algorithm to determine the necessary conditions for the integrability of the systems. This algorithm is then applied to a quadratic subfamily of the systems to analyze its integrability and linearizability. Our findings offer insights into the integrability properties of three-dimensional polynomial systems.
\smallskip

\noindent {\bf Math Subject Classification (2020).} 34C07; 37G15; 68W30

\smallskip


\section{Introduction}

The integrability of polynomial differential equations is a fundamental problem in the qualitative theory of differential equations. An efficient  approach for studying local integrability of nonlinear systems around equilibrium points involves transforming the original system into a more analytically tractable form using the theory of normal forms. Originating from the works of Poincar\'e, Dulac, and Lyapunov, this theory employs coordinate transformations to simplify the system, facilitating the analysis of its integrability and linearizability.

Consider an $n$-dimensional system of ODEs
\begin{equation}\label{firstSystem}
    \dot{x} = A x + X(x), \quad x \in \mathbb{C}^n,
\end{equation}
where $A$ is a diagonal matrix with eigenvalues $\lambda = (\lambda_1, \dots, \lambda_n),$ all of which are different from zero, $x = (x_1, \dots, x_n)^\top, X(x) = (X_1(x), \dots, X_n(x))^\top, X_i$ are series with complex coefficients which start with terms of degree greater than or equal to two. The  collection of eigenvalues  $(\lambda_ {1},\dots,\lambda _ {n})$ of $A$
is said to be resonant if there is a relation of the form
\begin{equation}\label{resonant}
 \alpha _ {1} \lambda _ {1} + \dots + \alpha _ {n} \lambda _ {n}  - \lambda_k=0
\end{equation}
for some  $  k \in \{ 1, \dots, n \} $,
with  $  \alpha_ {i} \in \mathbb{N}_0= \mathbb{N} \cup \{ 0 \} $,
$  \sum _ {i=1}  ^ {n} \alpha _ {i} \geq  2 $.

It is known (see e.g., \cite{Bib}) that any formal invertible change of coordinates of the form
\begin{equation}\label{substitutionNormal}
    x = y + h(y) = y + \sum_{j=2}^{\infty} h_j(y),
\end{equation}
with $h_j(y)$ being a vector of polynomials of degree $j,$ brings system (\ref{firstSystem}) to a system of a similar form, that is,
\begin{equation}\label{secondSystem}
    \dot{y} = A y + G(y), \quad G(y) = \sum_{j=2}^{\infty} G_j(y),
\end{equation}
for some vectors $G_j(y)$ of polynomials of degree $j$. For $\alpha=(\alpha_1, \dots, \alpha_n) \in \mathbb{N}_{0}^{n},$ define $|\alpha| = \alpha_1 + \cdots + \alpha_n.$ Let $x^{\alpha} = x_1^{\alpha_1} \cdots x_n^{\alpha_n}.$ Consider the monomial $x^{\alpha},$ for $|\alpha| > 1,$ in $e_k X(x),$ where $e_k$ is the $n$-dimensional row unit vector for all $k \in \{ 1, 2, \dots, n \}.$ This monomial  is called \textit{resonant} if $\lambda $ and $\alpha$ satisfy \eqref{resonant}. Similarly, the monomial $y^{\alpha} = y_1^{\alpha_1} \cdots y_n^{\alpha_n}$ in $e_k G(y)$ is resonant if (\ref{resonant}) holds. System (\ref{secondSystem}) is said to be in the \textit{Poincar\'e--Dulac normal form}, or simply in the normal form, if $G(y)$ includes only resonant terms. According to the Poincar\'e--Dulac theorem \cite{Poincare, Dulac2}, system (\ref{firstSystem}) can be transformed into a normal form by a substitution of the form  (\ref{substitutionNormal}). This transformation is referred to as a normalization or a normalizing transformation and is not unique.  However, the normalization containing only non-resonant terms is unique. It is called the distinguished normalizing transformation, and the corresponding normal form is called the distinguished normal form.

The following theorem of Siegel \cite{Sie} gives a criterion for convergence of the normalizing transformation in the case when
the eigenvalues of the matrix $A$ are non-resonant.

\begin{theorem} \label{th_S}
Assume that there are constants $C > 0, \sigma  > 0$
such that for all {non-negative} integer tuples
$\alpha =(\alpha _1, \dots, \alpha_n ), $
$\sum \alpha_i > 1$
\begin{equation}\label{cond_S}
|\sum_{i=1}^n \alpha_i \lambda_i - \lambda_j |\ge C (\alpha_1+\dots+\alpha_n)^{-\sigma}.
\end{equation}
Then there is a convergent transformation to normal form.
\end{theorem}

It is clear that in the case of Siegel's theorem the normal form is
\begin{equation} \label{lin_A}
\dot x =A x.
\end{equation}

The next theorem of Pliss \cite{Pliss} treats the case when the set of eigenvalues is resonant, but the normal form is still linear.

\begin{theorem}\label{th_P}
Assume that:

(a) the nonzero elements among the $
\sum_{i=1}^n m_i \lambda_i -\lambda_k$ satisfy condition \eqref{cond_S};

(b) some formal normal form of \eqref{firstSystem} is linear system \eqref{lin_A}.

Then there is a convergent transformation of system \eqref{firstSystem} to normal form \eqref{lin_A}.
\end{theorem}

It was proved by Bruno \cite{Br1,Br2} that if a certain condition (called the condition $\omega$) is satisfied and there is a normal form of \eqref{firstSystem} of the shape
\begin{equation} \label{lin_A_Br}
\dot x = (1+g(x)) A x,
\end{equation}
where $g(x)$ is a scalar function, then there is a convergent transformation to normal form \eqref{lin_A_Br}.

It was shown in \cite{W91} that for any $A$ the algebra of polynomial first integrals of \eqref{lin_A} is finitely generated. Moreover,  if the $\mathbb{Z}$-module $\mathcal{R}$ spanned by all non-negative integer solutions $(k_1, \dots , k_n)$ of $\sum_{i=1}^n
k_i \lambda_i = 0$ has rank $d$, then there are exactly $d$ independent polynomial first integrals (and also exactly $d$ independent formal first integrals) for system \eqref{lin_A}. It yields that the nonlinear system  \eqref{firstSystem} can have at most $d$ independent first integrals in a small neighborhood of the origin (see e.g., \cite{LPW}).

In the case when $d=n-1$ it is shown that that system \eqref{firstSystem} is (locally) completely integrable if it admits  $n-1$ independent first integrals in a neigborhood of the origin (otherwise it is said that it is partially integrable \cite{DRZ}). It was proved by Zhang \cite{Z,Z_b} that if some formal normal form of \eqref{firstSystem} admits $n-1$ independent formal first integrals, then \eqref{firstSystem} admits a convergent transformation to the distinguished  normal form, and there exist $n-1$ independent analytic first integrals. Moreover, in this case the distinguished normal form is of  the shape \eqref{lin_A_Br}. Rescaling the time in \eqref{lin_A_Br} by $1+g(x)$, we obtain a linear system, so in the case of complete integrability system \eqref{firstSystem} is orbitally linearizable.

The problem of complete integrability of three-dimensional polynomial systems has been extensively studied (see \cite{AR, Aziz, Aziz1, Aziz2,HHR} and references therein). It appears the case when the rank of $\mathcal{R}$ is smaller than $n-1$ has attracted much less attention. Recently, a system  of the form \eqref{firstSystem} with $A={\rm{diag}}(1,\zeta, \dots \zeta^{n} )$, where $n$ is a prime  number and $\zeta$ is the $n$-th primitive root of unity,  has been studied in \cite{GJR}. Clearly, in such case the algebra of polynomial first integrals of the linear system
\begin{equation} \label{Zlin}
\dot { x}=  {\rm{diag}}(1,\zeta, \dots \zeta^{n} ) { x}
\end{equation}
is generated by the single monomial
$x_1 x_2 \cdots x_n$. So, in this case the nonlinear system \eqref{firstSystem} with
$A={\rm{diag}}(1,\zeta, \dots \zeta^{n})$,
\begin{equation} \label{SYS}
\dot { x}=  {\rm{diag}}(1,\zeta, \dots \zeta^{n} ) { x}+X(x),
\end{equation}
can have at most one independent local analytic first integral,  which can be chosen in the form
\begin{equation}\label{int_Phi}
\Phi({\bf x})= x_1 x_2 \cdots x_n+h.o.t.
\end{equation}
Thus the study of local integrability of system \eqref{SYS} represents the simplest generalization of the Poincar\'e center problem from the two-dimensional case to the $n$-dimensional case (more detailed discussions about the center problem can be found in \cite{RS}).

In this paper we study system \eqref{SYS} in the case $n=3$. We limit our consideration to the polynomial systems, which, without loss of generality we write in the form

\begin{equation}\label{3D}
    \begin{aligned}
        \dot x_1 &= x_1 + \sum_{(p,q,r) \in S} a_{pqr}x_1^{p+1}x_2^{q}x_3^{r} = P(x_1,x_2,x_3), \\
        \dot x_2 &= \zeta \left( x_2 + \sum_{(p,q,r) \in S} b_{rpq}x_1^{r}x_2^{p+1}x_3^{q} \right) = Q(x_1,x_2,x_3), \\
        \dot x_3 &= \zeta^2 \left( x_3 + \sum_{(p,q,r) \in S} c_{qrp}x_1^{q}x_2^{r}x_3^{p+1} \right) = R(x_1,x_2,x_3),
    \end{aligned}
\end{equation}
where
\begin{equation} \label{setS}
S = \{ (p_j, q_j, r_j) \mid p_j + q_j + r_j \geq 1, j = 1, \dots, l \}  \subset \mathbb{N}_{-1}\times\mathbb{N}_{0}\times\mathbb{N}_{0}
\end{equation}
is a finite set and $\mathbb{N}_{-1}=\mathbb{N}_0\cup \{-1\}$. Furthermore, $x_1, x_2$ and $x_3$ are complex variables, the coefficients of $P, Q$ and $R$ are complex parameters, and $\zeta^3 = 1, \zeta \neq 1$. For the rest of this paper, by the integrability of system \eqref{3D} (or its subfamily) we mean existence of an analytic first integral of the form \eqref{int_Phi} for \eqref{3D} (or for the subfamily). We note that, by Theorem \ref{th_P}, if there is a formal transformation which linearizes system \eqref{3D}, then there is also a convergent transformation which linearizes the system. Then the product of the coordinate functions of the linearilizing transformation is an analytic first integral of the form \eqref{int_Phi} (with $n=3$). That is, formal or analytic linearizability of system \eqref{3D} yields integrability of the system.

In the present paper, we first give a criterion for convergence of the normal form of system \eqref{3D} and study the relation of the normal form and the integrability. Then, we propose an efficient algorithm to compute the necessary conditions of integrability of system \eqref{3D} and apply it to study the integrability and linearizability of a quadratic subfamily of systems \eqref{3D}.

\section{Normal forms and integrability of system \eqref{3D}}

It is easy to see that the normal form of system (\ref{3D}) can be written as
\begin{equation}\label{normalForm}
    \begin{aligned}
        \dot{y_1} =& y_1 \left ( 1 + \sum_{k \in \mathbb{N}} Y_1^{(k+1,k,k)} y_1^k y_2^k y_3^k \right ) = y_1 + y_1 Y_1(y_1 y_2 y_3), \\
        \dot{y_2} =& y_2 \left ( \zeta + \sum_{k \in \mathbb{N}} Y_2^{(k,k+1,k)} y_1^k y_2^k y_3^k \right )= \zeta y_2 + y_2 Y_2(y_1 y_2 y_3), \\
        \dot{y_3} =& y_3 \left ( \zeta^2 + \sum_{k \in \mathbb{N}} Y_2^{(k,k,k+1)} y_1^k y_2^k y_3^k \right ) = \zeta^2 y_3 + y_3 Y_3(y_1 y_2 y_3).
    \end{aligned}
\end{equation}

The following lemma shows that small divisors are not encountered in the normalization of system \eqref{3D}.

\begin{lemma}\label{lem1}
Let $\kappa=(\kappa_1,\kappa_2,\kappa_3) = (1, \zeta, \zeta^2)$ and $(\alpha, \kappa)$ denote the standard scalar product of vectors $\alpha$ and $\kappa$. Then
\begin{equation*}
 | (\alpha, \kappa) - \kappa_m | \geq 1,
 \end{equation*}
 for all $\alpha \in \mathbb{N}_{0}^{3}$ and $m \in \{ 1, 2, 3 \}$ such that $(\alpha, \kappa) - \kappa_m \neq 0$.
\end{lemma}
\begin{proof}
We will prove this lemma only for $m = 1,$ since the other two cases are similar. We have $(\alpha, \kappa) - \kappa_m \neq 0$ for $\alpha \neq (k+1, k, k)$ for $m=1$. Suppose $\alpha = (\alpha_1, \alpha_2, \alpha_3) \neq (k+1, k, k), k \in \mathbb{N}_{0}.$ Then  $(\alpha, \kappa) - \kappa_1 = \alpha_1 - 1 + \zeta \alpha_2 + \zeta^2 \alpha_3.$ Since $\zeta^2 + \zeta + 1 = 0,$ we can write $\zeta$ as $\zeta = -\frac{1}{2} + \frac{\sqrt{3}}{2} i$ and $\zeta^2 = -\frac{1}{2} - \frac{\sqrt{3}}{2} i$. Then,
\begin{equation*}
(\alpha, \kappa) - \kappa_1 = \alpha_1 - 1 + \zeta \alpha_2 + \zeta^2 \alpha_3 = \alpha_1 - 1 - \frac{1}{2}(\alpha_2 + \alpha_3) + \frac{\sqrt{3}}{2}(\alpha_2 - \alpha_3) i,
\end{equation*}
so Re$((\alpha, \kappa) - \kappa_1) = \alpha_1 - 1 - \frac{1}{2}(\alpha_2 + \alpha_3)$ and Im$((\alpha, \kappa) - k_1) = \frac{\sqrt{3}}{2}(\alpha_2 - \alpha_3).$

Suppose first that the imaginary part is equal to zero. This means that $\alpha_2 = \alpha_3$. Then $| (\alpha, \kappa) - \kappa_m | = |\alpha_1 - 1 - \alpha_2|.$ Since $\alpha_1 - 1 - \alpha_2$ is an integer and $\alpha_1 \neq \alpha_2 + 1,$ we have that $| (\alpha, \kappa) - \kappa_m | \geq 1.$

Let us now consider the case when the real part is equal to zero. This means that $| (\alpha, \kappa) - \kappa_1 | = \frac{\sqrt{3}}{2}|\alpha_2 - \alpha_3|.$ Since the real part is equal to zero, it holds that $\alpha_1 = 1 + \frac{1}{2}(\alpha_2 + \alpha_3) \in \mathbb{N}.$ This implies that $\frac{\sqrt{3}}{2}|\alpha_2 - \alpha_3| \geq \sqrt{3} > 1.$

Finally, consider the case when both the real and imaginary part are different from zero. If $\alpha_2 - \alpha_3 = 1,$ then the real part is $\alpha_1 - 1 - \alpha_3 - \frac{1}{2} = \alpha_1 - \alpha_3 - \frac{3}{2}.$ In this case we have
\begin{equation*}
 | (\alpha, \kappa) - \kappa_1 | =
 \sqrt{ (\alpha_1 - \alpha_3 - \frac{3}{2})^2 + \frac{3}{4}} \geq
        1.
\end{equation*}
We obtain the same result if $\alpha_2 - \alpha_ 3 = -1.$

If $|\alpha_2 - \alpha_3| \geq 2,$ then
\begin{equation*}
| (\alpha, \kappa) - \kappa_1 |  >\sqrt{3} > 1.
\end{equation*}
This concludes the proof.
\end{proof}

The next theorem gives a condition which ensures the convergence of the distinguished normal form.
\begin{theorem}
If $ Y_1  + Y_2 + Y_3 \equiv 0 $ then the distinguished normal form (\ref{normalForm}) of system \eqref{3D} is convergent.
\end{theorem}
\begin{proof}
In view of Lemma \ref{lem1}, by \cite[Theorem 3.2]{Bib} (or by \cite[Theorem 2.3.13]{RS}) the distinguished normal form of \eqref{3D} is convergent if there exists a positive constant $d$ such that for all $\alpha$ and $\beta$ in $\mathbb{N}_{0}^{3}$ for which $2 \leq |\beta| \leq |\alpha|-1, \alpha - \beta + e_m \in \mathbb{N}_{0}^{3}$ for all $m \in \{ 1, 2, 3 \},$ and $$(\alpha - \beta, \kappa) = 0,$$ it holds that
\begin{equation*}
\left | \sum_{j=1}^{3} \beta_{j} Y_{j}^{(\alpha - \beta + e_j)}  \right | \leq d |(\beta, \kappa)| \sum_{j=1}^{3} \left | Y_{j}^{(\alpha - \beta + e_j)} \right |.
\end{equation*}

In our case $\alpha = (k+1, k, k), (k, k+1, k), (k, k, k+1),$ for $j = 1, 2, 3$, respectively. But then $\beta = (m+1,m,m), (m,m+1,m), (m,m,m+1)$ for $j = 1, 2, 3$, respectively, where $k$ and $m$ are natural numbers and $k \geq m.$ Let $\alpha  = (k+1,k,k).$ Then $\beta = (m+1,m,m)$ and we have

\begin{equation*}
\begin{aligned}
&\Big | \sum_{j=1}^{3} \beta_{j} Y_{j}^{(\alpha - \beta + e_j)}  \Big| = \left | (m+1) Y_1^{(k-m+1,k-m,k-m)} + m  Y_2^{(k-m,k-m+1,k-m)} +  m Y_3^{(k-m,k-m,k-m+1)} \right | \\
= &\left | Y_1^{(k-m+1,k-m,k-m)} \right | \leq \sum_{j=1}^{3} \left | Y_{j}^{(\alpha - \beta + e_j)} \right | = | (\beta, \kappa) | \sum_{j=1}^{3} \left | Y_{j}^{(\alpha - \beta + e_j)} \right |,
\end{aligned}
\end{equation*}
because $(\beta, \kappa) = (m+1) + m \zeta + m \zeta^2 = 1.$ Similarly we obtain the same result for $\alpha = (k,k+1,k), \beta = (m,m+1,m)$ for which we have $|(\beta,\kappa)| = | \zeta | = 1,$ and $\alpha = (k,k,k+1), \beta = (m,m,m+1),$ for which  we have $|(\beta,\kappa)| = | \zeta^2 | = 1.$
\end{proof}

The next statement presents a generalization of some results of \cite{Bib,RS} obtained for the two-dimensional case to the case of three-dimensional system \eqref{3D}.
\begin{theorem}\label{th_2}
For system (\ref{3D}), we have the following statements.

1)  If system (\ref{3D}) has a formal first integral $\Phi(x)$ of the form \eqref{int_Phi} then
$Y_1 + Y_2 + Y_3 \equiv 0$ and the distinguished normal form is convergent.

2) If $Y_1 + Y_2 + Y_3 \equiv 0,$ then (\ref{3D}) has a  convergent first integral of the form \eqref{int_Phi}.
\end{theorem}
\begin{proof}
The second statement is obvious, so we prove only the first one. Let $H(y) = y + h(y)$ be the normalizing transformation.
Then
\begin{equation} \label{F}
F = \Phi \circ H  = \sum \limits_{(\alpha_1,\alpha_2,\alpha_3)} F^{(\alpha_1,\alpha_2,\alpha_3)} y_1^{\alpha_1} y_2^{\alpha_2} y_3^{\alpha_3} = y_1 y_2 y_3 + \cdots
\end{equation}
is a formal first integral of  the normal form yielding
\begin{equation}\label{normalFormEQ}
\begin{aligned}
 &y_1 \frac{\partial F}{\partial y_1}(y_1,y_2,y_3) + \zeta y_2 \frac{\partial F}{\partial y_2}(y_1,y_2,y_3) + \zeta^2 y_3 \frac{\partial F}{\partial y_3}(y_1,y_2,y_3) \\
 = &- y_1 \frac{\partial F}{\partial y_1}(y_1,y_2,y_3)Y_1(y_1y_2y_3) - y_2 \frac{\partial F}{\partial y_2}(y_1,y_2,y_3)Y_2(y_1y_2y_3) - y_3 \frac{\partial F}{\partial y_3}(y_1,y_2,y_3)Y_3(y_1y_2y_3).
\end{aligned}
\end{equation}

From \eqref{F} and \eqref{normalFormEQ} we have
\begin{equation}\label{RHS}
\begin{aligned}
 &\sum \limits_{(\alpha_1,\alpha_2,\alpha_3)} (\alpha_1 + \zeta \alpha_2 + \zeta^2 \alpha_3) F^{(\alpha_1,\alpha_2,\alpha_3)} y_1^{\alpha_1} y_2^{\alpha_2} y_3^{\alpha_3} \\
 = &- \left [ \sum \limits_{(\alpha_1,\alpha_2,\alpha_3)} \alpha_1 F^{(\alpha_1,\alpha_2,\alpha_3)} y_1^{\alpha_1} y_2^{\alpha_2} y_3^{\alpha_3} \right ] \left [ \sum \limits_{k=1}^{\infty} Y_1^{(k+1,k,k)} y_1^k y_2^k y_3^k \right ] \\
 = &- \left [ \sum \limits_{(\alpha_1,\alpha_2,\alpha_3)} \alpha_2 F^{(\alpha_1,\alpha_2,\alpha_3)} y_1^{\alpha_1} y_2^{\alpha_2} y_3^{\alpha_3} \right ] \left [ \sum \limits_{k=1}^{\infty} Y_2^{(k,k+1,k)} y_1^k y_2^k y_3^k \right ] \\
  = &- \left [ \sum \limits_{(\alpha_1,\alpha_2,\alpha_3)} \alpha_3 F^{(\alpha_1,\alpha_2,\alpha_3)} y_1^{\alpha_1} y_2^{\alpha_2} y_3^{\alpha_3} \right ] \left [ \sum \limits_{k=1}^{\infty} Y_3^{(k,k,k+1)} y_1^k y_2^k y_3^k \right ].
\end{aligned}
\end{equation}
By induction it is easy to see that if $\alpha_1 + \zeta \alpha_2 + \zeta^2 \alpha_3 \neq 0$, then the coefficient $F^{(\alpha_1, \alpha_2, \alpha_3)}$ in $F(y)$ is equal to zero. Thus, $F(y_1,y_2,y_3) = f(y_1y_2y_3).$

But, for  $F(y_1,y_2,y_3) = f(y_1y_2y_3)$, we have
\begin{equation}
y_1 \frac{\partial F}{\partial y_1} = y_1 y_2 y_3 f ' (y_1 y_2 y_3), \quad y_2 \frac{\partial F}{\partial y_2} = y_1 y_2 y_3 f ' (y_1 y_2 y_3) \quad \text{and} \quad y_3 \frac{\partial F}{\partial y_3} = y_1 y_2 y_3 f ' (y_1 y_2 y_3),
\end{equation}
so denoting $y_1 y_2 y_3 = w,$ from  (\ref{normalFormEQ})
we obtain
    \begin{equation*}
        0 \equiv -wf'(w)Y_1(w) - wf'(w)Y_2(w) - wf'(w)Y_3(w) = -wf'(w)(Y_1(w) + Y_2(w) + Y_3(w))
    \end{equation*}
    yielding $Y_1(w) + Y_2(w) + Y_3(w) \equiv 0.$

\end{proof}

\section{Integrability quantities}

Let
$$
\mathcal{X}=  P(x_1,x_2,x_3) \frac{\partial}{\partial x_1} + Q(x_1,x_2,x_3) \frac{\partial}{\partial x_2}+ R(x_1,x_2,x_3)\frac{\partial}{\partial x_3}
$$
be the vector field of system \eqref{3D}.
It is easy to see  that it is  possible to find a (formal) power series $\Psi$ of the form
\begin{equation}\label{FirstIntegral}
    \Psi(x_1,x_2,x_3) = v_{000} x_1 x_2 x_3 + \sum_{i+j+k \geq 4} v_{i-1,j-1,k-1} x_1^i x_2^j x_3^k,
\end{equation}
such that $v_{000}=1$,  $v_{k,k,k}=0$ for $k>0$ and
\begin{equation}\label{FirstIntegral2}
      \mathcal{X}(\Psi) =
       \sum_{k \geq 1}g_{k,k,k}x_1^{k+1} x_2^{k+1} x_3^{k+1}.
\end{equation}

The coefficients $g_{k,k,k}$ are called the {\it integrability quantities.} Obviously, they are polynomial in the parameters of system \eqref{3D}. From \eqref{FirstIntegral} we see that the integrability quantities $g_{k,k,k}$ and the coefficients $v_{k_1,k_2,k_3}$ of \eqref{FirstIntegral2} can be computed recursively by the formulae
\begin{equation}\label{rec v_k}
    \begin{aligned}
        &v_{k_1,k_2,k_3} = -\frac{1}{k_1+k_2\zeta+k_3\zeta^2} \biggl( \sum_{(p,q,r) \in S}(k_1+1-p)a_{pqr}v_{k_1-p,k_2-q,k_3-r} \\
        &+ \zeta \sum_{(p,q,r) \in S}(k_2+1-p)b_{rpq}v_{k_1-r,k_2-p,k_3-q}
        + \zeta^2 \sum_{(p,q,r) \in S}(k_3+1-p)c_{qrp}v_{k_1-q,k_2-r,k_3-p} \biggr)
    \end{aligned}
\end{equation}
and
\begin{equation}\label{g_k rec}
    \begin{aligned}
        &g_{kkk} = \sum_{(p,q,r) \in S}(k+1-p)a_{pqr}v_{k-p,k-q,k-r} + \zeta \sum_{(p,q,r) \in S}(k+1-p)b_{rpq}v_{k-r,k-p,k-q} + \\
        & \quad \zeta^2 \sum_{(p,q,r) \in S}(k+1-p)c_{qrp}v_{k-q,k-r,k-p}.
    \end{aligned}
\end{equation}

The following algorithm is a direct implementation of the formulae \eqref{rec v_k}
and \eqref{g_k rec}.

Calculating any of the polynomials $v_{k_1,k_2,k_3}$ or $g_{kkk}$ is a difficult computational problem because the number of terms in these polynomials grows so fast. To make a better understanding of the algebraic structure and properties of the integrability quantities $g_{k,k,k}$ and the coefficients $v_{k_1,k_2,k_3}$, we introduce the following notations.  It is also helpful in  developing  another algorithm for computing the integrability quantities $g_{k,k,k}$.

We order the set $S$ given by \eqref{setS}
and write the ordered set as $S = \{ (p_1,q_1,r_1), \dots, (p_l, q_l, r_l) \}.$ We then order the parameters accordingly as $$(a_{p_1q_1r_1}, \dots, a_{p_lq_lr_l}, b_{r_1p_1l_1}, \dots, b_{r_l,p_l,q_l}, c_{q_1r_1p_1}, \dots, c_{q_lr_lp_l}),$$ so that any monomial in
the parameters of system \eqref{3D} has the form
$$a_{p_1q_1r_1}^{\nu_1} \cdots a_{p_lq_lr_l}^{\nu_l} b_{r_1p_1q_1}^{\nu_{l+1}} \cdots b_{r_lp_lq_l}^{\nu_{2l}} c_{q_1r_1p_1}^{\nu_{2l+1}} \cdots c_{q_lr_lp_l}^{\nu_{3l}},$$ for some $\nu = (\nu_1, \dots, \nu_{3l}) \in \mathbb{N}_{0}^{3l}.$ To simplify the notation
for $\nu \in \mathbb{N}_{0}^{3l}$, we denote  the above monomial by $[\nu]$,
\begin{equation*}
    [\nu] \stackrel{\text{def}}{=} \quad  a_{p_1q_1r_1}^{\nu_1} \cdots a_{p_lq_lr_l}^{\nu_l} b_{r_1p_1q_1}^{\nu_{l+1}} \cdots b_{r_lp_lq_l}^{\nu_{2l}} c_{q_1r_1p_1}^{\nu_{2l+1}} \cdots c_{q_lr_lp_l}^{\nu_{3l}}.
\end{equation*}
We will write just $\mathbb{C}[a,b,c]$ instead of $\mathbb{C}[a_{p_1q_1r_1}, \dots, a_{p_lq_lr_l},b_{r_1p_1q_1},\dots,b_{r_lp_lq_l},c_{q_1r_1p_1},\dots,c_{q_lr_lp_l}],$ and for $f \in \mathbb{C}[a,b,c],$ we  write it as  $ f = \sum_{\nu\in \mathrm{Supp}(f)}f^{(\nu)}[\nu],$ where Supp$(f)$ denotes those $\nu \in \mathbb{N}_{0}^{3l}$ such that the coefficient of $[\nu]$ in the polynomial $f$ is nonzero.

Let $L : \mathbb{N}_{0}^{3l} \rightarrow \mathbb{Z}^{3}$
be the map defined by
\begin{equation*}
    \begin{aligned}
        L(\nu) &= (L_1(\nu), L_2(\nu), L_3(\nu)) = \nu_1(p_1,q_1,r_1) + \cdots + \nu_l(p_l,q_l,r_l) +\\
        &\nu_{l+1}(r_1,p_1,q_1) + \cdots + \nu_{2l}(r_l,p_l,q_l) + \nu_{2l+1}(q_1,r_1,p_1) + \cdots + \nu_{3l}(q_l,r_l,p_l).
    \end{aligned}
\end{equation*}
\begin{definition}
    For $(i,j,k) \in \mathbb{N}_{-1} \times \mathbb{N}_{-1} \times \mathbb{N}_{-1},$ a polynomial $f = \sum_{\nu \in \mathrm{Supp}(f)}f^{(\nu)}[\nu]$ in $\mathbb{C}[a,b,c]$ is an $(i,j,k)$-polynomial if for every $\nu \in \mathrm{Supp}(f), L(\nu)=(i,j,k).$
\end{definition}

The next theorem gives some properties of the coefficients $v_{i,j,k}$ of series \eqref{FirstIntegral} and of the integrability quantities.
\begin{theorem} \label{thmKKK}
Let family (\ref{3D}) be given. There exists a formal series $\varPsi(x_1,x_2,x_3)$ of the form (\ref{FirstIntegral}) and polynomials $g_{111}, g_{222}, \dots$ $\in \mathbb{C}[a,b,c]$ such that
\begin{enumerate}
\item $ \mathcal{X} \varPsi  = g_{111}(x_1x_2x_3)^2 + g_{222}(x_1x_2x_3)^3 + g_{333}(x_1x_2x_3)^4 + \cdots $,
\item for every triple  $(i,j,k) \in \mathbb{N}_{-1}^{3}, i+j+k \geq 0, v_{ijk} \in \mathbb{C}[a,b,c]$, and $v_{ijk}$ is an $(i,j,k)$-polynomial,
\item for every $k \geq 1, v_{kkk} = 0,$ and
\item for every $k \geq 1, g_{kkk} \in \mathbb{C}[a,b,c],$ and $g_{kkk}$ is a $(k,k,k)$-polynomial.
\end{enumerate}
\end{theorem}
The proof is similar to the proof of Theorem 3.4.2 in \cite{RS}. Using Theorem \ref{thmKKK} and formulae \eqref{rec v_k} and \eqref{g_k rec} we can derive an efficient algorithm for computing the integrability quantities as follows.

For any $\nu \in \mathbb{N}_{0}^{3l}$ define $V(\nu)$ recursively, with respect to $|\nu| = \nu_1 + \nu_2 + \cdots + \nu_{3l},$ setting
\begin{center}
$V((0,0, \dots, 0)) = 1;$
\end{center}
and then for $\nu  \neq (0,0, \dots, 0)$,
\begin{center}
$V(\nu) = 0 \quad \text{if} \quad L_1(\nu) = L_2(\nu) = L_3(\nu),$
\end{center}
and when $L_i(\nu)$  are not all the same for $i=1,2,3$,
\begin{equation}\label{V(v)}
\begin{aligned}
V(\nu) = & -\frac{1}{L_1(\nu)+\zeta L_2(\nu)+\zeta^2 L_3(\nu)}\times \\
& \left( \sum_{j=1}^{l}\widetilde{V}(\nu_1, \dots, \nu_j - 1, \dots, \nu_{3l})(L_1(\nu_1, \dots, \nu_j - 1, \dots, \nu_{3l})+1) \right. \\
& \quad + \zeta \sum_{j=l+1}^{2l}\widetilde{V}(\nu_1, \dots, \nu_j - 1, \dots, \nu_{3l})(L_2(\nu_1, \dots, \nu_j - 1, \dots, \nu_{3l})+1)   \\
& \quad + \left.\zeta^2 \sum_{j=2l+1}^{3l}\widetilde{V}(\nu_1, \dots, \nu_j - 1, \dots, \nu_{3l})(L_3(\nu_1, \dots, \nu_j - 1, \dots, \nu_{3l})+1) \right),
\end{aligned}
\end{equation}
where
\begin{equation}\label{Vti}
    \widetilde{V}(\eta) =
    \begin{cases}
        V(\eta) & \text{if $\eta \in \mathbb{N}_{0}^{3l}$} \\
        0 & \text{otherwise}.
    \end{cases}
\end{equation}

The proof of the following theorem is similar to the proof of Theorem 3.4.5 in \cite{RS}, so we omit it.

\begin{theorem}\label{thmV(v)}
For a family of systems of the form (\ref{3D}), let $\Psi$ be the formal series of the form (\ref{FirstIntegral}) and let $ v_{k_1,k_2,k_3} $
and $g_{kkk}$ be the polynomials in $\mathbb{C}[a,b,c]$ given by \eqref{rec v_k} and \eqref{g_k rec}, respectively, which satisfy the conditions of Theorem \ref{thmKKK}. Then

1) for $\nu \in \mathrm{Supp}(v_{k_1,k_2,k_3}),$ the coefficient $v_{k_1,k_2,k_3}^{(\nu)}$ of $[\nu]$ in $v_{k_1,k_2,k_3}$ is $V(\nu)$;

2) for $\nu \in \mathrm{Supp}(g_{kkk}),$ the coefficient $g_{kkk}^{(\nu)}$ of $[\nu]$ in $g_{kkk}$ is
\begin{equation}
\begin{aligned}\label{g_kkk for alg}
g_{kkk}^{(\nu)} &= \sum_{j=1}^{l}\widetilde{V}(\nu_1, \dots, \nu_j - 1, \dots, \nu_{3l})(L_1(\nu_1, \dots, \nu_j - 1, \dots, \nu_{3l})+1) \\
\quad &+ \zeta \sum_{j=l+1}^{2l}\widetilde{V}(\nu_1, \dots, \nu_j - 1, \dots, \nu_{3l})(L_2(\nu_1, \dots, \nu_j - 1, \dots, \nu_{3l})+1)   \\
\quad &+ \zeta^2 \sum_{j=2l+1}^{3l}\widetilde{V}(\nu_1, \dots, \nu_j - 1, \dots, \nu_{3l})(L_3(\nu_1, \dots, \nu_j - 1, \dots, \nu_{3l})+1),
\end{aligned}
\end{equation}
where $\widetilde V$ is defined by \eqref{Vti}.
\end{theorem}

The results presented above yield our second algorithm that can be used for computing the integrability quantities for systems of the form  (\ref{3D}).

Correctness of the {\bf Algorithm 2} follows from Theorem \ref{thmV(v)}. In the following we present several examples of different sets (see (\ref{set-al})) to demonstrate the applicability and the computational efficiency of our two algorithms. The experimental results show the applicability and efficiency of our algorithms. All the experiments were made in Maple 18, running under a computer with Intel(R) Xeon(R) Silver 4210 CPU @ 2.20GHz.

\begin{equation}\label{set-al}
\begin{split}
S_1&=\{(2,0,0),(1,0,0),(0,1,0),(0,0,1)\},\\
S_2&=\{(1,0,0),(0,1,0),(0,0,1)\},\\
S_3&=\{(1,0,0),(0,0,1)\}.
\end{split}
\end{equation}

\begin{table}[htbp]
\begin{center}
\caption{Computational times (in seconds) of the two algorithms for different sets of $S$.}\label{Tab-1}
\smallskip
\begin{tabular}{|c|c|c||c||c||c||c|}
  \hline
\multirow{2}{*}{$k$} & \multicolumn{3}{|c|}{{\bf Algorithm 1}} & \multicolumn{3}{|c|}{{\bf Algorithm 2}} \\
\cline{2-7}
  & $S_1$ & $S_2$ & $S_3$ & $S_1$ & $S_2$ & $S_3$ \\ \hline
 1 & 0.015 & 0.015 &  0.012 & 0.744 & 0.505 &  0.049\\
 2  & 0.303 & 0.152 & 0.128 & 63.486 & 9.723 & 1.563\\
 3  & 3.769 & 2.135 & 0.607 & - & 304.620 & 10.650\\
 4  & 75.628 & 15.982 & 2.073 & - & - & 71.682\\
 5  & 3486.938 & 158.924 & 5.857 & - & - & 510.636\\ \hline
\end{tabular}
\end{center}
\end{table}

\begin{table}[htbp]
\begin{center}
\caption{Number of terms in the polynomial $g_{kkk}$.}\label{Tab-2}
\smallskip
\begin{tabular}{|c|c|c||c|}
  \hline
\multirow{2}{*}{$k$} & \multicolumn{3}{|c|}{\# of terms in $g_{kkk}$}  \\
\cline{2-4}
  & $S_1$ & $S_2$ & $S_3$ \\ \hline
 1 & 12 & 12 &  4 \\
 2  & 404 & 280 &  32 \\
 3  & 3644 & 1676 & 100 \\
 4  & 19142 & 6164 &  214\\
 5  & 74790 & 17572 &  388\\ \hline
\end{tabular}
\end{center}
\end{table}

\begin{remark}\label{rem-algo-1}
From Table \ref{Tab-1}, one can see that the {\bf Algorithm 2} can be applied to the computation of the integrability quantities for systems of small set of $S$ or low orders of $k$ (say, $k<4$). The {\bf Algorithm 1} is expected to have a better performance for systems of the form (\ref{3D}) with many parameters. The two algorithms are implemented in Maple, and the source code of the Maple programs is avaliable at {\textcolor{blue}{\url{https://github.com/Bo-Math/integrability}}}.
\end{remark}

\section{Integrability conditions for a family of quadratic systems}

The aim of this section is to look for integrable systems for the family of the following quadratic systems
\begin{equation}\label{example2}
    \begin{aligned}
        \dot x_1 &= x_1 + a_{100}x_1^2 + a_{010}x_1 x_2+a_{001}x_1 x_3, \\
        \dot x_2 &= \zeta (x_2 + b_{100}x_1x_2 + b_{010}x_2^2 + b_{001}x_2x_3), \\
        \dot x_3 &= \zeta^2 (x_3 + c_{100}x_1x_3 + c_{010}x_2x_3 + c_{001}x_3^2).
    \end{aligned}
\end{equation}
The set $S$ corresponding to this system is  $S = \{ (1,0,0), (0,1,0), (0,0,1)\}$.

A subset of integrable systems in a given family of systems \eqref{3D}, which is possible to detect without
computing the integrability quantities, is the set of systems admitting a reversible symmetry introduced in \cite{W,LPW}.

To define this symmetry we rewrite our system (\ref{3D}) as
\begin{equation}\label{zX}
\dot{x} = \mathcal{Z} x + X(x) = F(x),
\end{equation}
where $\mathcal{Z} = \text{diag}[1, \zeta, \zeta^2]$.

\begin{definition}
We say that system (\ref{zX}) is \emph{$\zeta$-reversible} if
\begin{equation}\label{zetaRev}
    A F(x) = \zeta F(Ax),
\end{equation}
for $\zeta ^3 = 1, \zeta \neq 1$ and some matrix
\begin{equation}\label{mxA}
    A =
        \begin{pmatrix}
            0 & \alpha & 0 \\
            0 & 0 & \beta \\
            \gamma & 0 & 0 \\
        \end{pmatrix},
\end{equation}
where $\alpha \beta \gamma = 1.$ System (\ref{zX}) is called \emph{equivariant} if (\ref{zetaRev}) holds for $\zeta = 1$.
\end{definition}

The above definition is based on the results of \cite{W,LPW,GJR} (see also \cite{TM,WRZ}). The following result gives the information for system (\ref{3D}) to have a local analytic first integral in a neighborhood of the origin.

\begin{theorem}\label{th_5}
If system (\ref{3D}) is $\zeta$-reversible then it admits a local analytic first integral $\Phi$ of the form \eqref{int_Phi} in a neighborhood of the origin.
\end{theorem}
\begin{proof}
By Proposition 11 of \cite{LPW} and Proposition 2.5 of \cite{GJR}, any
$\zeta$-reversible system \eqref{3D} admits a formal first integral of the form \eqref{int_Phi}.
By Theorem \ref{th_2} the system has also a convergent first integral of the form \eqref{int_Phi}.
\end{proof}

According to Proposition 4.12 of \cite{GJR}, in order to find $\zeta$-reversible systems in family \eqref{example2}
one can compute the fourth elimination ideal $I_{\zeta}$ of the ideal
\begin{equation*}
\begin{aligned}
I^{(\zeta)} = \langle 1 - &w \alpha \beta \gamma, \zeta \alpha a_{100} - b_{010}, \zeta \beta a_{010} - b_{001}, \zeta \gamma a_{001} - b_{100}, \zeta \alpha b_{100} - c_{010}, \\
&\zeta \beta b_{010} - c_{001}, \zeta \gamma b_{001} - c_{100}, \zeta \alpha c_{100} - a_{010}, \zeta \beta c_{010} - a_{001}, \zeta \gamma c_{001} - a_{100} \rangle
\end{aligned}
\end{equation*}
in the ring $\mathbb{C}[w,\alpha,\beta,\gamma, a,b,c]$, so $I_{\zeta} \cap \mathbb{C}[w,\alpha,\beta,\gamma, a,b,c]$.

Computations with \texttt{Singular} \cite{GTZ2} over the field $\mathbb{Q}[\zeta]$ yield
\begin{equation*}
\begin{aligned}
I_{\zeta} = \langle &b_{010} b_{100}-a_{100} c_{010}, b_{001} b_{100}-a_{001} c_{100}, a_{010} b_{100}-c_{010} c_{100}, b_{001} b_{010}-a_{010} c_{001}, \\
&a_{001} b_{010} -c_{001} c_{010}, a_{100} b_{001}-c_{001} c_{100}, a_{010} a_{100}-b_{010} c_{100}, a_{001} a_{100}-b_{100} c_{001}, \\
&a_{001} a_{010}-b_{001} c_{010} \rangle.
\end{aligned}
\end{equation*}

Thus, the variety ${\bf V}(I_{\zeta})$ is the Zariski closure of the set of $\zeta$-reversible systems in the family \eqref{example2}
and by Theorem \ref{th_5} all systems from ${\bf V}(I_{\zeta})$ admit an analytic first integral of the form \eqref{int_Phi}.

To find another integrable systems in family \eqref{example2} we need to compute the integrability quantities and look for the irreducible decomposition of the variety that they generate. However, it involves extremely laborious computations, which we were not able to complete for full family \eqref{example2}. To facilitate the calculations, we consider the subfamily of \eqref{example2} setting $b_{001}=c_{100}=0$ and $b_{010}=1$. Under these conditions the system has the form
\begin{equation}\label{example}
\begin{aligned}
\dot x_1 &= x_1 + a_{100}x_1^2 + a_{010}x_1x_2+a_{001}x_1x_3, \\
\dot x_2 &= \zeta (x_2 + b_{100}x_1x_2 + x_2^2), \\
\dot x_3 &= \zeta^2 (x_3 + c_{010}x_2x_3 + c_{001}x_3^2).
\end{aligned}
\end{equation}

For system \eqref{example} using the Integrability Quantities Algorithm we calculated the first five integrability quantities $g_{111}, \dots, g_{555}$. The resulting expressions can be found at the website {\textcolor{blue}{\url{https://github.com/Bo-Math/integrability}}}. Using the routine \texttt{minAssGTZ} (which is based on the algorithm of \cite{GTZ}) we have found that the variety of $I = \langle g_{111}, g_{222}, g_{333}, g_{444}, g_{555} \rangle$ consists of the following nine irreducible components:
\begin{enumerate}
    \item[(1)] ${\bf V}(J_1),$ where $J_1 = \langle c_{001}, a_{001} \rangle$,
    \item[(2)] ${\bf V}(J_2),$ where $J_2 = \langle c_{010}+\zeta, a_{010} - 2\zeta -1, a_{001} a_{100}-\zeta a_{001} b_{100}+b_{100} c_{001} \rangle$,
    \item[(3)] ${\bf V}(J_3),$ where $J_3 = \langle c_{010} + \zeta, a_{001} + \zeta c_{001}, a_{010} a_{100}-(\zeta+1) a_{100} + \zeta b_{100} \rangle$,
    \item[(4)] ${\bf V}(J_4),$ where $J_4 = \langle b_{100}, a_{100} \rangle$,
    \item[(5)] ${\bf V}(J_5),$ where $J_5 = \langle b_{100}, a_{001} \rangle$,
    \item[(6)] ${\bf V}(J_6),$ where $J_6 = \langle a_{010} c_{001}-(\zeta+1) c_{001} c_{010}+(\zeta+1) a_{001}, a_{010} a_{100}-(\zeta+1) a_{100} c_{010}+(\zeta+1) b_{100}, a_{001} a_{100}-b_{100} c_{001} \rangle$,
    \item[(7)] ${\bf V}(J_7),$ where $J_7 = \langle c_{010}, a_{010} a_{100}+(\zeta+1) b_{100} \rangle$,
    \item[(8)] ${\bf V}(J_8),$ where $J_8 = \langle c_{010}, a_{001} \rangle$,
    \item[(9)] ${\bf V}(J_9),$ where $J_9 = \langle b_{100}, a_{010} - \zeta c_{010} \rangle$.
\end{enumerate}

The equations defining these varieties give the necessary conditions for integrability of system \eqref{example}. We prove that some of these conditions are also the sufficient conditions for integrability, in fact, the corresponding systems are linearalizable.

\begin{theorem}
If the parameters of system \eqref{example} belong to the varieties of one of the ideals $J_1, J_4,  J_5,  J_8 $ or $J_9$,
then the corresponding system is linearizable.
\end{theorem}
\begin{proof}
Case \textit{(1)}. In this case the system is
\begin{equation} \label{c1}
\begin{aligned}
\dot x_1 &= x_1(1 + a_{100} x_1 + a_{010} x_2),
\\
\dot x_2 &= \zeta x_2(1 + b_{100} x_1 + x_2),
\\
\dot x_3 &= \zeta^2 x_3(1 + c_{010} x_2),
\end{aligned}
\end{equation}
and it is a subcase of the system
\begin{equation} \label{c1-1}
\begin{aligned}
\dot x_1 &= x_1(1 + a_{100} x_1 + a_{010} x_2),
\\
\dot x_2 &= \zeta x_2(1 + b_{100} x_1 + b_{010} x_2),
\\
\dot x_3 &= \zeta^2 x_3(1 + c_{010} x_2).
\end{aligned}
\end{equation}
By Theorem 1 of \cite{petekromanovski}, there is a normal form \eqref{normalForm} of system \eqref{c1-1} such that
$$Y_1^{(k+1,k,k)},\quad Y_2^{(k,k+1,k)}, \quad Y_3^{(k,k,k+1)}  $$
are $(k,k,k)$-polynomials. But products of parameters of \eqref{c1-1} cannot form $(k,k,k)$-polynomials. Thus, the normal form is linear. By Lemma \ref{lem1} and the Theorem \ref{th_P} system \eqref{c1-1} is analytically linearizable.

Similar reasoning works in case \textit{(4)}.

Case \textit{(5)}. We first note that by \cite[Lemma 1]{Aziz1} the system
\begin{equation*}
\dot{X_1} = \lambda X_1, \quad \dot{X_2} = \nu X_2, \quad \dot{x_3} =  x_3(-\mu + \epsilon x_3 + A(X_1, X_2)),
\end{equation*}
where $\lambda, \mu, \nu \in \mathbb{Z}^+, \epsilon$ is a parameter and $A(X_1, X_2)$ is a function of $X_1$ and $X_2,$ is linearizable if there exist functions $\gamma (X_1, X_2)$ and $\xi(X_1,X_2)$ such that $\gamma(0,0) = 1$ and
\begin{equation}\label{cond}
\dot{\xi} - \mu \xi = \epsilon \gamma, \quad \dot{\gamma} = A \gamma.
\end{equation}
It is easy to see that the statement remains valid also in the case $\lambda=1, \nu=\zeta, \mu=-\zeta^2$. The system of this case is
\begin{equation}\label{c5}
\begin{aligned}
\dot x_1 &= x_1(1 + a_{100} x_1 + a_{010} x_2),
\\
\dot x_2 &= \zeta x_2(1 +  x_2),
\\
\dot x_3 &= \zeta^2 x_3(1 + c_{010} x_2 + c_{001} x_3).
\end{aligned}
\end{equation}

By Theorem \ref{th_S} the system of two first equations is linearizable, so after the linearizing transformations $x_1=X_1(x_1,x_2), \ x_2=X_2(x_1,x_2)$ we obtain the system
$$
\dot X_1=X_1, \quad \dot X_2=\zeta X_2, \quad \dot x_3 =  \zeta^2   x_3 (1 + c_{010} X_2 + c_{001} x_3).
$$

We use the above statement, where
$A(X_1, X_2) = \zeta^2 c_{010} X_2$, so \eqref{cond} has the form
\begin{equation}\label{co11}
\dot{\xi} + \zeta^2 \xi = \zeta^2 c_{001} \gamma, \quad  \dot{\gamma} = \zeta^2 c_{010} X_2 \gamma.
\end{equation}
Following the proof of \cite[Lemma 1]{Aziz1}, it can be easily verified that there are convergent series $\xi$ and $\gamma$ satisfying \eqref{co11}, so the system is linearized. Indeed, if we make a substitution $\gamma = e^{\theta}$, where $\theta = \theta(X_1, X_2),$ the second equation becomes $\dot{\theta} = \zeta^2 c_{010} X_2$. The solution to this differential equation is $\theta(X_1,X_2) = \zeta c_{010} X_2$, which is convergent and subsequently $\gamma$ is convergent. We see that $\gamma$ is dependent only on $X_2$ and from $\gamma = e^{\zeta c_{010} X_2}$ we obtain
\begin{equation*}
\gamma(X_1, X_2) = e^{\zeta c_{010} X_2} = 1 + \sum_{j \geq 1} \frac{\zeta^j c_{010}^j}{j!} X_2^j.
\end{equation*}
Let $\xi = \xi(X_1, X_2) = \sum_{i+j \geq 0} e_{ij} X_1^i X_2^j$. If we substitute these expressions into the first differential equation of (\ref{co11}), we get
\begin{equation}\label{co12}
\begin{aligned}
\sum_{i+j \geq 0} (i+\zeta j + \zeta^2)e_{ij} X_1^i X_2^j = \zeta^2 c_{001} + \sum_{j \geq 1} \frac{\zeta^{j-1} c_{001}c_{010}^j}{j!}X_2^j.
\end{aligned}
\end{equation}
Equating the coefficients of each monomial $X_1^i X_2^j$ on both sides, we immediately see that $e_{ij} = 0$, for any $i \geq 1,$ and $$e_{0j} = \frac{\zeta^{j-1} c_{001} c_{010}^j}{j! (\zeta j + \zeta^2)}.$$
Thus, the expression for $\xi = \xi (X_1, X_2)$ is
$$\xi (X_1, X_2) = \sum_{j \geq 0} \frac{\zeta^{j-1} c_{001} c_{010}^j}{j!(\zeta j + \zeta^2)} X_2^j.$$ It can be easily proved that $|\zeta j + \zeta^2| \geq 1$, for any $j \geq 1.$  This implies that function $\xi$ is convergent.

Case \textit{(8)}. In this case the system is
\begin{equation}\label{c8}
\begin{aligned}
\dot x_1 &= x_1(1 + a_{100} x_1 + a_{010} x_2),
\\
\dot x_2 &= \zeta x_2 (1 + b_{100} x_1 + x_2),
\\
\dot x_3 &= \zeta^2 x_3(1 + c_{001} x_3).
\end{aligned}
\end{equation}
By Theorem \ref{th_S}, the system of two first equations is linearizable at the origin in the $ (x_1,x_2)$ plane. The third equation is Darboux linearizable. Thus, system \eqref{c8} is analytically linearizable.

Case \textit{(9)}. The system of this case is
\begin{equation}\label{c9}
\begin{aligned}
\dot x_1 &= x_1(1 + a_{100} x_1 + \zeta c_{010} x_2  + a_{001} x_3),
\\
\dot x_2 &= \zeta  x_2(1 +  x_2),
\\
\dot x_3 &= \zeta^2   x_3(1 + c_{010} x_2 + c_{001} x_3).
\end{aligned}
\end{equation}
The system of the last two equations is linearizable, so these equations can be transformed to $\dot{X_2} = \zeta X_2$ and $\dot{X_3} = \zeta^2 X_3$, and after the substitutions we obtain the system
$$
\dot x_1 = x_1(1 + a_{100}x_1 + \zeta c_{010} X_2 + a_{001} X_3), \quad \dot X_2=\zeta X_2, \quad \dot X_3 =  \zeta^2 X_3.
$$
Using similar arguments as in case \textit{(5)} we see that the above system is linearizable.

\end{proof}

\begin{conjecture}
Systems whose parameters belong to the varieties of the ideals $J_2, J_3, J_6, J_7$ are integrable, but not linearizable.
\end{conjecture}

We do not have a proof of this conjecture, but computations of the normal forms up to order 7 indicate that it should hold true. It appears that in the case of the ideals $J_2$ and $J_7$, one of equations of the normal form is linear, whereas in the case of the ideals $J_3$ and $J_6$ all three equations of the normal form are nonlinear.

\section*{Acknowledgments}
Bo Huang's work is supported by the National Natural Science Foundation of China under Grant No. 12101032. Ivan Mastev and Valery Romanovski acknowledge the support of the Slovenian Research and Innovation Agency (core research program P1-0306).

\end{document}